\documentclass[12pt]{article}
\usepackage{graphicx, amssymb, amsthm, amsfonts, latexsym, epsfig, amsmath, amscd, amsbsy}
\usepackage[usenames]{color}
\usepackage{tikz}

\theoremstyle{plain}
\newtheorem{theorem}{Theorem}
\newtheorem{lem}[theorem]{Lemma}

\theoremstyle{definition}

\def\R{{\mathbb R}}
\def\N{{\mathbb N}}
\def\Z{{\mathbb Z}}
\def\T{{\mathbb T}}

\def\E{{\mathcal E}}

\def\T{{\mathcal T}}
\def\G{{\mathcal G}}

\def\K{{\mathcal K}}
\def\a{{\alpha}}
\def\b{{\beta}}

\newcommand{\la}{\overleftarrow}
\newcommand{\ra}{\overrightarrow}

\newcommand{\tl}{\triangleleft}

\newcommand{\h}{\hspace{-.1cm}}

\def\bar{\overline}
\def\cdot{\raisebox{2pt}{$\centerdot$}}
\def\lpinf{\hspace{.2cm}+^{\hspace{-.5cm}\infty}\hspace{.2cm}}
\def\rpinf{+^{\hspace{-.1cm}\infty}}

\begin{document}

\title{Symbolic dynamics for Lozi maps}
\author{M.\ Misiurewicz,
  S.\ \v{S}timac\thanks{Supported in part by the NEWFELPRO Grant
    No.~24 HeLoMa.}} \date{}

\maketitle

\begin{abstract}
In this paper we study the family of the Lozi maps $L_{a,b} : \R^2 \to
\R^2$, $L_{a,b} = (1 + y - a|x|, bx)$, and their strange attractors
$\Lambda_{a,b}$. We introduce the set of kneading sequences for the
Lozi map and prove that it determines the symbolic dynamics for that
map. We also introduce two other equivalent approaches.
\end{abstract}



\baselineskip=18pt

\section{Introduction}\label{sec:intro}
Symbolic dynamics and the Milnor -- Thurston kneading theory are very
powerful tools in studying one-dimensional dynamics of unimodal maps,
such as the tent maps, or the quadratic maps. One of the most
important ingredients in the kneading theory for unimodal maps is the
kneading sequence, which is defined as the itinerary of the critical
value. This symbol sequence is a complete invariant of the topological
conjugacy classes of unimodal maps with negative Schwarzian derivative
(when there is no periodic attractor). A key step in proving this fact
is that the set of all possible itineraries of such a map is
completely characterized by its kneading sequence. For a unimodal map
$f$ with the turning point $c$, restricted to an invariant interval $I
= [f^2(c),f(c)] \subseteq [0,1]$, called the core of $f$, this
characterization is as follows. A sequence $\ra x = (x_i)_{i \in
  \N_0}$ is an itinerary of some point $x \in I$ if and only if
$\sigma^n(\ra x) \preceq \ra k(f)$ for every $n \in \N_0$, where $\ra
k(f)$ is the kneading sequence of $f$, $\sigma$ is the shift map,
$\preceq$ is the parity-lexicographical ordering, and $\N_0$ is the
set of all non-negative integers ($\N$ will be the set of positive
integers). Besides this, the kneading theory plays a fundamental role
in various applications, for example in proving monotonicity of the
topological entropy for the family of skew tent maps, see~\cite{MV}.

On the other hand, no theory of comparable rigor for such `good
symbolic invariants' is currently available for general once-folding
mappings of the plane such as the H\'enon and the Lozi mappings. The
obstruction to constructing a similar theory for such mappings are
that they lack critical points in the usual sense (their Jacobians
never vanish) and that their dynamical space lack a natural order
which one-dimensional dynamics \emph{a priori} have.

In this paper we study the family of the Lozi maps $L_{a,b} : \R^2 \to
\R^2$, $$L_{a,b} = (1 + y - a|x|, bx)$$ and their strange attractors
$\Lambda_{a,b}$. We introduce the set of kneading sequences for the
Lozi map and prove that the set of all itineraries of points in
$\Lambda_{a,b}$ is completely characterized by the set of kneading
sequences of $L_{a,b}$. This characterization for the Lozi maps has
the same flavor as the Milnor -- Thurston characterization for the
unimodal maps mentioned above. The difference is that the Lozi map has
countably many kneading sequences and one needs criteria when to use
which sequence. We also introduce a \emph{folding pattern} and a
\emph{folding tree}, which can replace the set of kneading
sequences. They carry the same information as the set of kneading
sequences, coded in a different way.

The paper is organized as follows. In Section~\ref{sec:pre} we
summarize basic information about Lozi maps. In Sections~\ref{sec:def}
and~\ref{sec:ba} we define various notions used later in the paper. In
Section~\ref{sec:orders} we introduce orders which can partially
replace the natural orders on an interval and in the set of
itineraries, that work so well for unimodal interval maps. In
Section~\ref{sec:three} we present the three equivalent approaches to
coding the basic information about a Lozi map: the set of kneading
sequences, the folding pattern, and the folding tree. In
Section~\ref{sec:symbolic} we show how the set of kneading sequences
(or the folding pattern, or the folding tree) determine the symbolic
dynamics for a Lozi map.

\section{Preliminaries}\label{sec:pre}

The family of piecewise affine mappings $L_{a,b} = (1 + y - a|x|, bx)$
of the plane into itself was given by Lozi in 1978 \cite{L}. The
results of his numerical investigations for the values of parameters
$a = 1.7$ and $b = 0.5$ suggested the existence of a strange
attractor. Figure \ref{fig.LA} shows the strange attractor for Lozi's
original choice of parameters.

\begin{figure}[ht]
\centering
\includegraphics[width=7.0cm,height=6.0cm]{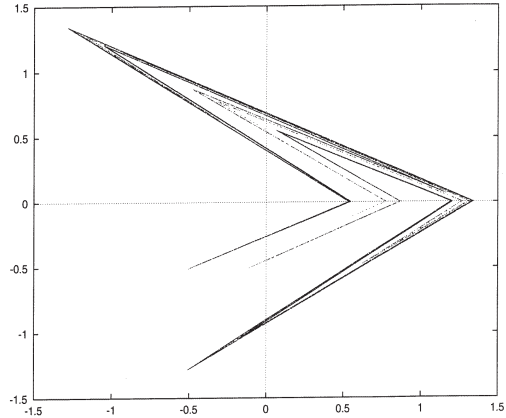}
\caption{The Lozi attractor for $a = 1.7$ and $b = 0.5$.}
\label{fig.LA}
\end{figure}

A mathematical justification for the existence of strange attractors
of the Lozi maps was given by the first author in 1980 \cite{M}. It
was proved there that the Lozi mappings have strange attractors for
$(a,b) \in S$, where the set $S$ is shown in Figure \ref{fig.S} (the
figure is a copy of a figure in \cite{M}) and is given by the
following inequalities: $b > 0$, $a\sqrt{2} > b+2$, $b <
\frac{a^2-1}{2a+1}$, $2a + b < 4$.\footnote{In several figures we use
  values of $a$ and $b$ that are not in this set, in order to get a
  better picture.}

\begin{figure}
	\centering
	\includegraphics[width=8.0cm,height=6.5cm]{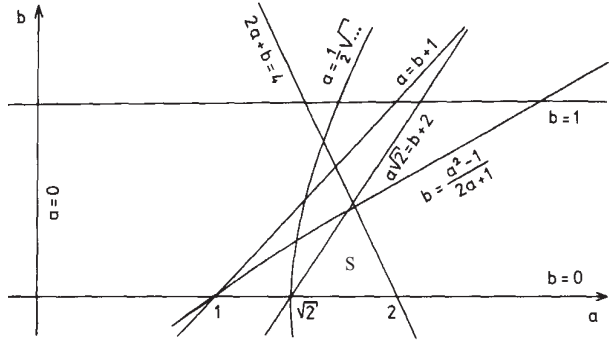}
	\caption{The set S.}
	\label{fig.S}
\end{figure}

Let $(a, b) \in S$ and, for simplicity, denote $L := L_{a,b}$. The map
$L$ is a homeomorphism which linearly maps the left half-plane onto
the lower one and the right one onto the upper one. There are two
fixed points: $X = (\frac{1}{1+a-b}, \frac{b}{1+a-b})$ in the first
quadrant and $Y = (-\frac{1}{a+b-1}, -\frac{b}{a+b-1})$ in the third
quadrant. They are hyperbolic. Note that the Lozi map $L$ is not
everywhere differentiable, and therefore its hyperbolic structure can
be understood only as the existence of a hyperbolic splitting at those
points at which it may exist (for which the derivative exists at the
whole trajectory). Recall, the \emph{stable}, respectively
\emph{unstable, manifold} of a fixed point $P$ (or a periodic point in
general), $W^s_P$, respectively $W^u_P$, is an invariant curve which
emanates from $P$,
$$W^s_P = \{ T : L^n(T) \to P \textrm{ as } n \to \infty \}, W^u_P =
\{ T : L^{-n}(T) \to P \textrm{ as } n \to \infty \}.$$ For the Lozi
map $L$ the stable and unstable manifolds are broken lines, and
therefore not the differentiable manifolds. The half of the unstable
manifold $W^u_X$ of the fixed point $X$ that starts to the right
intersects the horizontal axis for the first time at the point $Z =
(\frac{2+a+\sqrt{a^2+4b}}{2(1+a-b)}, 0)$. Let us consider the triangle
$\Delta$ with vertices $Z$, $L(Z)$ and $L^2(Z)$, see
Figure~\ref{fig.D}.
\begin{figure}[h]
	\begin{center}
	{
	\centering
	\begin{tikzpicture}
	\tikzstyle{every node}=[draw, circle, fill=white, minimum size=1.5pt, inner sep=0pt]
	\tikzstyle{dot}=[circle, fill=white, minimum size=0pt, inner sep=0pt, outer sep=-1pt]
	\node (n1) at (2*4/3,0) {};
	\node (n2) at (-2*4/3,2*2/3)  {};
	\node (n3) at (-2*2/3,-2*2/3)  {};
	\node (n4) at (2*4/9,2*2/9)  {};
	\node (n5) at (0,2*1/3)  {};

	\tikzstyle{every node}=[draw, circle, fill=white, minimum size=0.1pt, inner sep=0pt]
	\node (n6) at (-3,0) {};
	\node (n7) at (3,0)  {};
	\node (n8) at (0,2)  {};
	\node (n9) at (0,-2)  {};

	\foreach \from/\to in {n1/n2,n1/n3, n2/n3, n6/n7,n8/n9}
	\draw (\from) -- (\to);

	\node[dot, draw=none, label=above: $Z$] at (2*4/3,0) {};
	\node[dot, draw=none, label=above: $L^{-1}(Z)$] at (0.8,1.6) {};
	\node[dot, draw=none, label=above: $L(Z)$] at (-2.3,2.1) {};
	\node[dot, draw=none, label=above: $L^2(Z)$] at (-1,-1) {};
	\node[dot, draw=none, label=above: $X$] at (2*4/9,2*2/9) {};

	\end{tikzpicture}}
	\end{center}
	\caption{The triangle $\Delta$.}
	\label{fig.D}
\end{figure}
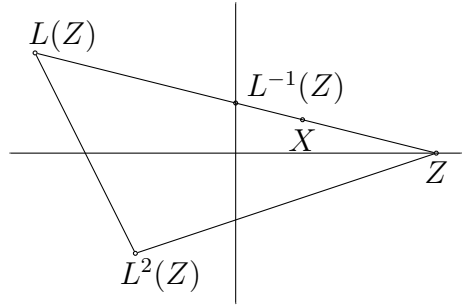
In the mentioned paper \cite{M} it was proved that $L(\Delta) \subset
\Delta$ and moreover, that
$$\Lambda = \bigcap_{n=0}^{\infty} L^n(\Delta)$$ is the strange
attractor. Moreover, $L|_{\Lambda}$ is topologically mixing, and
$\Lambda$ is the closure of $W^u_X$. Recall that an \emph{attractor}
is the set that is equal to the intersection of images of some its
neighborhood, and such that the mapping restricted to this set is
topologically transitive. An attractor is called \emph{strange} if it
has a fractal structure.

We code the points of $\Lambda$ in the following standard way. To a
point $P = (P_x, P_y)\in \Lambda$ we assign a bi-infinite sequence
$\bar p = \dots p_{-2} \, p_{-1} \cdot \, p_0 \, p_1 \, p_2 \dots$
such that

$p_n = -1$ if $P^n_x \le 0$ and

$p_n = +1$ if $P^n_x \ge 0$, where $L^n(P) = (P^n_x, P^n_y)$.\\
The dot shows where the 0th coordinate is. Moreover, to simplify
notation, we use just symbols $+$ and $-$ instead of $+1$ and $-1$.

A bi-infinite symbol sequence $\bar q = \dots q_{-2} \, q_{-1} \cdot
\, q_0 \, q_1 \, q_2 \dots$ is called \emph{admissible} if there is a
point $Q \in \Lambda$ such that $\bar q$ is assigned to $Q$. We will
call this sequence an \emph{itinerary} of $Q$. Obviously, some points
of $\Lambda$ have more than one itinerary. We denote the set of all
admissible sequences by $\Sigma_\Lambda$. It is a metrizable
topological space with the usual product topology. Since the
half-planes that we use for coding, intersected with $\Lambda$, are
compact, the space $\Sigma_\Lambda$ is compact. From the hyperbolicity
of $L$ it follows that for every admissible sequence $\bar q$ there
exists only one point $Q\in\Lambda$ with this itinerary. The detailed
proof can be extracted from the paper of Ishii~\cite{I}. Because of
this uniqueness, we have a map $\pi:\Sigma_\Lambda\to\Lambda$, such
that $\bar q$ is an itinerary of $\pi(\bar q)$. Clearly, $\pi\circ
L=\sigma\circ\pi$.

In fact, Ishii was identifying the itineraries of the same point and
he proved that the shift map in the quotient space is conjugate (via
the map induced by $\pi$) with $L$ on $\Lambda$. In our setup, this
just means that $\pi$ is continuous, and is a semiconjugacy between
$(\Sigma_\Lambda,\sigma)$ and $(\Lambda,L)$.

\section{Definitions}\label{sec:def}

Let us consider the unstable manifold of the fixed point $X$, $W^u_X$.
It is an image of the real line under a map which is continuous and
one-to-one. For simplicity, we denote it by $R := W^u_X$. We denote by
$R^+$ the half of $R$ that starts at the fixed point $X$ and goes to
the right and intersects the horizontal axis for the first time at the
point $Z$. By $R^-$ we denote the other half of $R$ that also starts
at the fixed point $X$ and goes to the left and intersects the
vertical axis for the first time at the point $L^{-1}(Z)$. Let $[A, B]
\subset R$ denote an arc of $R$ with boundary points $A$ and $B$, and
let $(A, B) = [A, B] \setminus \{ A, B \}$. For a point $P \in R$ and
$\epsilon > 0$ let
$$(P - \epsilon, P + \epsilon) = \{ Q \in R : d(P,Q) < \epsilon \}
\subset R,$$ where $d(P,Q)$ denotes length of the arc $[P, Q]$.

We introduce an ordering $\tl$ on $R$ in the following natural way:
For $P, P' \in R^+$ we say that $$P \tl P' \ \textrm{ if } \ [X, P]
\subset [X, P'].$$ For $P, P' \in R^-$ we say that $$P \tl P'
\ \textrm{ if } \ [X, P] \supset [X, P'].$$ Also, if $P \in R^-$ and
$P' \in R^+$, we say that $P \tl P'$. Note that $L(R^+) = R^-$ and
$L(R^-) = R^+$.

{\bf Gluing points.} We call a point $G = (G_x, G_y) \in R$ a
\emph{gluing point} if $G_x = 0$ and there is no $k \in \N$ such that
$G^{-k}_x = 0$. Let us denote the set of all gluing points by ${\G}$.

{\bf Turning points.} We call a point $T = (T_x, T_y) \in R$ a
\emph{turning point} if $T = L(G)$ for some $G \in \G$. In this case
$T_y = 0$. Let us denote the set of all turning points by ${\T}$. Let
us denote by $\widehat{\T} := \{ L^{j}(T) : T \in {\T}, j \in \N
\}$ the set of \emph{postturning points}.

{\bf Basic points.} Let
$$\E = {\G} \cup {\T} \cup \widehat{\T} = \{ L^{j}(G) : G \in
{\G}, j \in \N_0 \},$$ We call the points of $\E$ the \emph{basic
  points}.

We can think of $R$ in two ways. The first one is $R$ as a subset of
the plane. The second way is to ``straighten'' it and consider it the
real line. Our order on $R$ is the natural order when we use the
second way. Note that the topology in $R$ is different in both cases.

\begin{lem}\label{l:discrete}
The set $\E$ is discrete, and therefore countable.
\end{lem}

\begin{proof}
Note that $L(Z) \tl L^{-1}(Z) \tl Z$ are three consecutive basic
points, where $L^{-1}(Z)$ is a gluing point and $Z$ is a turning
point. Note also that $L|_R$ is expanding and $R =
\bigcup_{i=0}^{\infty}L^i([L(Z), Z])$. If $P$ and $Q$ are two
consecutive points of $\E$, then between $L(P)$ and $L(Q)$ there is at
most one point of $\E$. Therefore by induction we see that in
$L^i([L(Z), Z])$ there are only finitely many points of $\E$. This
proves that $\E$ is discrete.
\end{proof}

The set of all gluing points is discrete. Let
$${\G} = \{ G_n, n \in \Z \},$$ where $G_0 = L^{-1}(Z)$ and $G_n \tl
G_{n+1}$ for every $n \in \Z$. The set of all turning points is also
discrete. Let
$${\T} = \{ T_n, n \in \Z \},$$ where $T_0 = Z$ and $T_n \tl T_{n+1}$
for every $n \in \Z$. We have $L(G_n) = T_{-n}$ for every $n \in \Z$.
Let
$$\E = \{ E_n, n \in \Z \},$$ where $E_0 = G_0$ and $E_i \tl E_{i+1}$
for every $i \in \Z$. Note that $E_1 = T_0$, $E_2 = G_1$, $E_{-1} =
L(T_0)$, and $E_{-2} = G_{-1}$.

We call the arcs between two consecutive basic points $[E_i,
  E_{i+1}]$, $i \in \Z$, the \emph{basic arcs}.

We denote the $j$th image of the $k$th turning point by $T^{j+1}_k :=
L^{j}(T_k)$, $j \in \N_0$ (note that $T^1_k = T_k$). Also, we denote
the $x$- and $y$-coordinates of the basic points as follows: $G_k =
(G_{k,x},G_{k,y})$, $T_k = (T_{k,x},T_{k,y})$, $T^j_k =
(T^j_{k,x},T^j_{k,y})$.

Let us now consider itineraries of points of $R$. Let $\Sigma_R$
denote the set of all itineraries of all points of $R$. For a sequence
$\bar p = (p_i)_{i \in \Z} \in \Sigma_R$ and $n \in \Z$, we will call
the left-infinite sequence $\la p_{\h n} = \dots p_{n-2} \, p_{n-1} \,
p_n$ a \emph{left tail} of $\bar p$ and the right-infinite sequence
$\ra p_{\h n} = p_n \, p_{n+1} \, p_{n+2} \dots$ a \emph{right tail}
of $\bar p$. We will call a finite sequence $W = w_1 \dots w_k$ a
\emph{word} and denote its length by $|W|$, $|W| = k$. We will denote
an infinite to the right (respectively, left) sequence of $+$s by
$\rpinf$ (respectively, $\lpinf$).

The itinerary of the fixed point $X$ (which is in the first quadrant)
is $\bar x = \lpinf \cdot \rpinf$. Since $R$ is the unstable
manifold of $X$, for every point $P \in R$ and its itinerary $\bar p$,
there is $n \in \Z$ such that $\la p_{\h n} = \lpinf$. Therefore, 
there exists the smallest integer $k > n$ such that $P_x^k = 0$ and 
$R$ crosses the $y$-axis at $L^k(P)$.
Also $P_y^{k+1} = 0$, $P_x^{k+1} > 0$ and there exists $\delta > 0$
such that for all points $Q \in (P - \delta , P + \delta)$ we have $0
< Q_x^{k+1} \le P_x^{k+1}$. In other words, $R$ makes a turn at the
point $L^{k+1}(P)$. Moreover, if there exists also $l > k$ such that
$P_x^l = 0$, then $R$ does not cross $y$-axis at $L^{l}(P)$, but also
makes a turn at $L^{l}(P)$. In other words, there exists $\epsilon >
0$ such that for every point $Q \in (P - \epsilon , P + \epsilon)$, $Q
\ne P$, and its itinerary $\bar q$, we have either $Q_x^l < P_x^l$ and
hence $q_l = -$, or $Q_x^l > P_x^l$ and hence $q_l = +$. Therefore,
instead of considering two itineraries of $P$, we consider only one,
with $p_l = -$ in the first case and $p_l = +$ in the second case.
This amounts to removing from $\Sigma_R$ isolated points. The
remaining part of $\Sigma_R$ will be denoted $\Sigma_R^e$. In such a
way every point $P \in R$ has at most two itineraries, and if there
are two of them, then they differ at one coordinate $k$, and then
$L^k(P)$ is a gluing point. In such a case we will sometimes write
$p_k=\pm$.

The set $\Sigma_R^e$ has a very useful property.

\begin{lem}\label{approx}
Assume that $\bar p\in\Sigma_R^e$ and $n\in\N$. Then there is $\bar
q\in\Sigma_R^e$ such that $p_{-n},\dots p_n=q_{-n}\dots q_n$ and $\bar
q$ is the only itinerary of $\pi(\bar q)$.
\end{lem}

\begin{proof}
Since the map $L$ is linear on the left and right half-planes, the set
of points of $\Delta$ that have an itinerary $\bar q$ such that
$p_{-n},\dots p_n=q_{-n}\dots q_n$ is a closed convex polygon, perhaps
degenerate. Therefore, in a neighborhood of $P=\pi(\bar p)$ in $R$ (in
the topology of the real line), its intersection with $R$ is a closed
interval $J$, perhaps degenerate to a point. However, if $J$ is only one
point, then the itinerary $\bar p$ is isolated and hence one of the removed ones, that
is, it does not belong to $\Sigma_R^e$. Therefore, $J$ is a
nondegenerate closed interval.

The set $R$ intersects the $y$-axis only at countable number of
points. Therefore the set of points $Q\in R$ such that $L^k(Q)$
belongs to the $y$-axis for some $k\in\Z$, is countable. Thus, there
are points $Q\in J$ such that for every $k\in\Z$ the point $L^k(Q)$
does not belong to the $y$-axis. Such $Q$ has only one itinerary. This
completes the proof.
\end{proof}

\section{Basic arcs and coding}\label{sec:ba}

Recall that we call the arcs between two consecutive basic points $[E_i,
  E_{i+1}]$, $i \in \Z$, the \emph{basic arcs}.

Let $P \in R$ be a point and let $\bar p$ be its itinerary. Note that
$$P \in (G_0, T_0) = (E_0, E_1) \ \Rightarrow \ \la p_{\h 0} =
\lpinf.$$ Since always $T_{0,x} > 0$ and $T^2_{0,x} < 0$ (that
follows easily from the assumptions on $a$ and $b$), we have
$L([E_0, E_1]) \ni G_0$. This implies that
$$L([E_0, E_1]) = [T^2_0, G_0] \cup [G_0, T_0] = [E_{-1},
  E_0] \cup [E_0, E_1]$$ and $$P \in (E_{-1}, E_0)
\ \Rightarrow \ \la p_{\h 0} = \lpinf -.$$

Consider now $L([T^2_0, G_0]) = [T_0, T^3_0]$. If $T^3_{0,x} \ge 0$
then $[T_0, T^3_0]$ does not contain any gluing point (both boundary
points are in the right half-plane) and hence $[T_0, T^3_0]$ is a
basic arc
$$L([E_{-1}, E_0]) = [E_{1}, E_{2}] \subset R^+$$ and $$P \in (E_1,
E_2) \ \Rightarrow \ \la p_{\h 0} = \lpinf -+.$$

If  $T^3_{0,x} < 0$ then  $G_1 \in [T_0, T^3_0]$ and
$$L([E_{-1}, E_0]) = [T_0, G_1] \cup [G_1, T^3_0] = [E_{1}, E_{2}]
\cup [E_{2}, E_3] \subset R^+.$$ Hence $$P \in (E_1, E_2)
\ \Rightarrow \ \la p_{\h 0} = \lpinf -+,$$
$$P \in (E_2, E_3) \ \Rightarrow \ \la p_{\h 0} = \lpinf --$$
(see Figure \ref{fig.BP1}).

\begin{figure}[h]
	{
	\centering
	\begin{tikzpicture}
	\tikzstyle{every node}=[draw, circle, fill=white, minimum size=1pt, inner sep=0pt]
	\tikzstyle{dot}=[circle, fill=white, minimum size=0pt, inner sep=0pt, outer sep=-1pt]
	\node (n1) at (5*4/3,0) {};
	\node (n2) at (-5*4/3,5*2/3)  {};
	\node (n3) at (-5*2/3,-5*2/3)  {};
	\node (n4) at (5*5/9,0)  {};
	\node (n5) at (-5*5/6,-5*1/3) {};
	\node (n6) at (-5*19/24,-5*5/12)  {};
	\node (n7) at (5*1/36,5*5/18)  {};
	\node (n8) at (5*61/51,0)  {};
	\node (n9) at (5*65/75,0)  {};

	\node (n10) at (5*4/9,5*2/9)  {};

	\tikzstyle{every node}=[draw, circle, fill=white, minimum size=0.1pt, inner sep=0pt]
	\node (n16) at (-7,0) {};
	\node (n17) at (7,0)  {};
	\node (n18) at (0,4)  {};
	\node (n19) at (0,-4)  {};

	\foreach \from/\to in {n1/n2,n1/n3, n2/n4,n5/n4, n8/n7,n7/n9, n3/n8,n9/n6, n16/n17,n18/n19}
	\draw (\from) -- (\to);

	\node[dot, draw=none, label=above: $T_0$] at (5*4/3,0) {};
	\node[dot, draw=none, label=above: $G_0$] at (0,2.35) {};
	\node[dot, draw=none, label=above: $T^2_0$] at (-5*4/3,5*2/3) {};
	\node[dot, draw=none, label=above: $G_{-1}$] at (0,1.2) {};
	\node[dot, draw=none, label=above: $T_{-1}$] at (5*5/9,0) {};
	\node[dot, draw=none, label=above: $G_{-2}$] at (0,0) {};
	\node[dot, draw=none, label=above: $T^4_0$] at (-5*5/6,-5*1/5) {};
	\node[dot, draw=none, label=above: $T^5_0$] at (-5*19/24,-5*5/12) {};
	\node[dot, draw=none, label=above: $G_{3}$] at (0,-1) {};
	\node[dot, draw=none, label=above: $T_2$] at (5*65/75,0) {};
	\node[dot, draw=none, label=above: $T^2_{-1}$] at (0,1.7) {};
	\node[dot, draw=none, label=above: $T_1$] at (5*61/51,0) {};
	\node[dot, draw=none, label=above: $T^3_0$] at (-5*2/3,-5*2/3) {};
	\node[dot, draw=none, label=above: $G_{1}$] at (0,-2.2) {};
	\node[dot, draw=none, label=above: $G_{2}$] at (0,-1.6) {};
	\node[dot, draw=none, label=above: $X$] at (2.25,1.7) {};

	\node[dot, draw=none, label=above: $I_{-}$] at (-3,3) {};
	\node[dot, draw=none, label=above: $I_{\emptyset}$] at (3.3,1.4) {};
	\node[dot, draw=none, label=above: $I_{-+-}$] at (-3,2.3) {};
	\node[dot, draw=none, label=above: $I_{-++}$] at (1.2,0.9) {};
	\node[dot, draw=none, label=above: $I_{--+}$] at (1.2,0.4) {};
	\node[dot, draw=none, label=above: $I_{---}$] at (-2,-0.4) {};
	\node[dot, draw=none, label=above: $I_{-+}$] at (3.3,-1.1) {};
	\node[dot, draw=none, label=above: $I_{--}$] at (-1.3,-2.6) {};

	\end{tikzpicture}}

	\unitlength=10mm

	\begin{picture}(20,4)(3,4.4)

        \put(12.21, 13.56){\circle*{0.1}}
	\put(8.3,6){\circle*{0.1}}
	\put(8.3,5.3){$X$}

	\put(3,6){\line(1,0){14}}

	\put(3,5.9){\line(0,1){0.2}}\put(2.7,5.3){$T_0^4$}\put(3,6.3){$_{---}$}
	\put(4,5.9){\line(0,1){0.2}}\put(3.6,5.3){$G_{-2}$}\put(4,6.3){$_{--+}$}
	\put(5,5.9){\line(0,1){0.2}}\put(4.6,5.3){$T_{-1}$}\put(5,6.3){$_{-++}$}
	\put(6,5.9){\line(0,1){0.2}}\put(5.6,5.3){$G_{-1}$}\put(6,6.3){$_{-+-}$}
	\put(7,5.9){\line(0,1){0.2}}\put(6.7,5.3){$T^2_0$}\put(7.4,6.3){$_{-}$}
	\put(8,5.9){\line(0,1){0.2}}\put(7.5,5.3){$G_0$}\put(8.4,6.3){$_{\emptyset}$}
	\put(9,5.9){\line(0,1){0.2}}\put(8.8,5.3){$T_0$}\put(9.2,6.3){$_{-+}$}
	\put(10,5.9){\line(0,1){0.2}}\put(9.7,5.3){$G_1$}\put(10.2,6.3){$_{--}$}
	\put(11,5.9){\line(0,1){0.2}}\put(10.7,5.3){$T_0^3$}\put(11,6.3){$_{-+--}$}
	\put(12,5.9){\line(0,1){0.2}}\put(11.7,5.3){$G_2$}\put(12,6.4){$_{-+-+}$}
	\put(13,5.9){\line(0,1){0.2}}\put(12.8,5.3){$T_1$}\put(13,6.3){$_{-+++}$}
	\put(14,5.9){\line(0,1){0.2}}\put(13.6,5.3){$T_{-1}^2$}\put(14,6.4){$_{--++}$}
	\put(15,5.9){\line(0,1){0.2}}\put(14.8,5.3){$T_2$}\put(15,6.3){$_{---+}$}
	\put(16,5.9){\line(0,1){0.2}}\put(15.7,5.3){$G_3$}\put(16,6.4){$_{----}$}
	\put(17,5.9){\line(0,1){0.2}}\put(16.7,5.3){$T_0^5$}

	\put(9,6.5){$\overbrace{\qquad \qquad \ \ \, }$}
	\put(3,5.1){$\underbrace{\qquad \qquad \ \ \, }$}
	\put(5,5.1){$\underbrace{\qquad \qquad \ \  \,}$}
	\put(11,6.6){$\overbrace{\qquad \qquad \ \ \, }$}
	\put(15,6.6){$\overbrace{\qquad \qquad \ \ \, }$}
	\put(7.5,6.5){\line(0,1){0.5}}
	\put(7.5,7){\line(1,0){2.5}}
	\put(10,7){\vector(0,-1){0.2}}

	\put(9.5,5.6){\line(0,-1){1}}
	\put(9.5,4.6){\line(-1,0){3.5}}
	\put(6,4.6){\vector(0,1){0.2}}

	\put(10.5,5.6){\line(0,-1){1.2}}
	\put(10.5,4.4){\line(-1,0){6.5}}
	\put(4,4.4){\vector(0,1){0.4}}

	\put(6.5,6.5){\line(0,1){0.7}}
	\put(6.5,7.2){\line(1,0){5.5}}
	\put(12,7.2){\vector(0,-1){0.3}}

	\put(5.5,6.5){\line(0,1){0.9}}
	\put(5.5,7.4){\line(1,0){8}}
	\put(13.5,7.4){\vector(0,-1){0.9}}

	\put(4.5,6.5){\line(0,1){1.1}}
	\put(4.5,7.6){\line(1,0){10}}
	\put(14.5,7.6){\vector(0,-1){1}}

	\put(3.5,6.5){\line(0,1){1.3}}
	\put(3.5,7.8){\line(1,0){12.5}}
	\put(16,7.8){\vector(0,-1){0.9}}

	\thicklines

	\put(7,5.99){\line(1,0){1}}
	\put(7,6.01){\line(1,0){1}}
	\put(9,5.99){\line(1,0){2}}
	\put(9,6.01){\line(1,0){2}}

	\end{picture}

\caption{Basic points and basic arcs for $a=1.75$ and $b=0.5$.}
\label{fig.BP1}
\end{figure}

If in addition $T^4_{0,x} < 0$ (and $T_{-1,x} > 0$, which holds for
all $(a,b) \in S$, implying $T_{i,x} > 0$ for all $i \in \Z$), we have
$$L^2([E_{-1}, E_0]) = L([T_0, G_1]) \cup L([G_1, T_0^3]) = [T_0^2,
T_{-1}] \cup [T_{-1}, T_0^4] \subset R^-.$$ Since $T_0^2$ and
$T_0^4$ are in the left half-plane and $T_{-1}$ is in the right
half-plane, we have $[T_0^2, T_{-1}] \ni G_{-1}$ and $[ T_{-1}, T_0^4]
\ni G_{-2}$, implying
\begin{eqnarray*}
	L^2([E_{-1}, E_0]) &=& [T_0^2, G_{-1}] \cup [G_{-1}, T_{-1}] \cup
	[T_{-1}, G_{-2}] \cup [G_{-2}, T_0^4]\\ &=& [E_{-1}, E_{-2}] \cup
	[E_{-2}, E_{-3}] \cup [E_{-3}, E_{-4}] \cup [E_{-4}, E_{-5}] \subset
	R^-.
\end{eqnarray*}
Hence
$$P \in (E_{-1}, E_{-2}) \ \Rightarrow \ \la p_{\h 0} = \lpinf -+-,$$
$$P \in (E_{-2}, E_{-3}) \ \Rightarrow \ \la p_{\h 0} = \lpinf -++,$$
$$P \in (E_{-3}, E_{-4}) \ \Rightarrow \ \la p_{\h 0} = \lpinf --+,$$
$$P \in (E_{-4}, E_{-5}) \ \Rightarrow \ \la p_{\h 0} = \lpinf ---$$
(see Figure \ref{fig.BP1}).

Continuing this procedure, we can code basic arcs $[E_i, E_{i+1}]$, 
$i \in \Z$, with finite words of $-$s and $+$s in
the following way. Since the map $L$ on $R$ is expanding, for every
basic arc $J$ as above there exists the smallest $n$ such that
$L^{-n-1}(J)\subset [G_0,T_0]$. For points $P\in J$ we have then
$\la p_{\h 0} = \lpinf a_{-n} \dots a_{-1} a_0 = \lpinf \a.$ In
this case we will use the following notation: $I_{\a} := [E_i,
  E_{i+1}]$. In particular, $[E_0, E_1] = I_\emptyset$.

Observe that for every $m \in \N, m \ge 1$, all basic arcs of
$L^{m-1}(I_-)$ are coded by words of length $m$. Moreover, if $m$ is
even, $L^{m-1}(I_-) \subset R^+$, and if $m$ is odd, $L^{m-1}(I_-)
\subset R^-$.

\section{Orders}\label{sec:orders}

Let us look at the points of $R$ and their itineraries. On $R$ (when we
think about it as the real line) we have the natural order $\tl$. We
want to introduce a corresponding order in the itineraries. Since the
situation is similar as for unimodal interval maps, we start with the
usual \emph{parity-lexicographical order}. Let $\ra p = p_0p_1 \dots$,
$\ra q = q_0q_1 \dots$ be two different right-infinite sequences or
finite words. Let $k \in \N_0$ be the smallest integer such that $p_k
\ne q_k$. Then $\ra p \prec \ra q$ if either $p_0 \dots p_{k-1}$ is
even (contains an even number of $+$s) and $p_k < q_k$, or $p_0 \dots
p_{k-1}$ is odd (contains an odd number of $+$s) and $q_k < p_k$.
Here, if $k = 0$, then $p_0 \dots p_{k-1}$ is the empty word, and $-1
< +1$, that is $- \prec +$ (if $p_k = \pm$, or $q_k = \pm$, then by
convention $- \prec \pm \prec +$). (We also allow that $\ra p = p_0p_1
\dots p_n$ is a finite word and $\ra q = q_0q_1 \dots$ is a
right-infinite sequence, or vice versa, and in this case we say that
$\ra p \prec \ra q$ if $p_0p_1 \dots p_n \prec q_0q_1 \dots q_n$.)
While this does not work if the lengths of $\ra p$ and $\ra q$ are
different and one of them is the beginning of the other one, we will
not encounter such situations.

When we want to define some reasonable order in $\Lambda$, we have to
use similar ideas as in the relativity theory (in the space-time).
Recall that we have a forward invariant unstable cone of directions.
(see~\cite{M}). We will call two distinct points $P,Q\in\Lambda$
\emph{comparable} if the direction of the straight line containing 
$P$ and $Q$ belongs to
this cone.

\begin{lem}\label{l:cone}
Assume that $P,Q\in\Lambda$ are comparable. Then $\ra p_{\h 0} \prec
\ra q_{\h 0}$ if and only if the $x$-coordinate of $P$ is smaller than
the $x$-coordinate of $Q$.
\end{lem}

\begin{proof}
The map $L$ expands distances on straight lines whose direction is in
the unstable cone. The expansion factor is at least a constant larger
than 1 dependent only of $a$ and $b$. Moreover, the unstable cone is
mapped to itself by the derivative of $L$ and $\Lambda$ is bounded
(see~\cite{M}). Therefore, there is $n\in N_0$ such that $L^n$ is
linear on the straight line segment with endpoints $P$, $Q$ and one of
the points $L^n(P)$, $L^n(Q)$ lies in the left half-plane, while the
other one lies in the right half-plane. The smallest such $n$ is
exactly the smallest $n\ge 0$ for which $p_n\ne q_n$. For each $i$
between 0 and $n-1$ the order in the $x$-direction between the points
$L^{i+1}(P)$ and $L^{i+1}(Q)$ is opposite to the analogous order
between $L^i(P)$ and $L^i(Q)$ if $p_i=+$, and the same if $p_i=-$.
Comparing this with the definition of the parity-lexicographical
order gives the result described in the lemma.
\end{proof}

Note that the direction of the local unstable manifold of $X$ is in
the unstable cone, so by the invariance of the unstable cone we get
that the direction of every straight line segment contained in $R$ is
contained in the invariant cone. In particular, we get immediately the
following result.

\begin{lem}\label{l:plorder}
Assume that $P,Q\in[G_0,T_0]$ and $P\ne Q$. Then $\ra p_{\h 0} \prec
\ra q_{\h 0}$ if and only if $P\tl Q$.
\end{lem}

This allows us to relate the orders $\prec$ and $\tl$.

Observe that a point $P\in R$ belongs to $[G_0,T_0]$ if and only if
$\la p_{\h 0}=\lpinf$. Let us define the \emph{generalized
  parity-lexicographical order} on the set $\Sigma_R$ in the following
way. Let $\bar p, \bar q \in \Sigma_R$, $\bar p \ne \bar q$. Let $n
\in \N$ be a positive integer such that $\la p_{\h -n} = \la q_{\h -n}
= \lpinf$. Then $\bar p \prec \bar q$ if either
\begin{enumerate}
	\item $n$ is even and $\ra q_{\h -n+1} \prec \ra p_{\h -n+1}$, or
	\item $n$ is odd and $\ra p_{\h -n+1} \prec \ra q_{\h -n+1}$.
\end{enumerate}
By the definition of the parity-lexicographical order and since $L$
reverses orientation on $[G_0,T_0]$, this order is well defined (it
does not depend on the choice of $n$).

{}From this definition and Lemma~\ref{l:plorder} we get immediately
the following result.

\begin{lem}\label{l:order of points}
Let $P, Q \in R$ be two different points and let $\bar p, \bar q$ be
their itineraries. Then $P \tl Q$ if and only if $\bar p \prec \bar
q.$
\end{lem}

Another straightforward consequence of Lemma~\ref{l:cone} is the
following lemma. It compares right tails of itineraries of points of
$R$ to the corresponding right tails of itineraries of the turning points.

\begin{lem}\label{l:lessknead}
Assume that $P,Q\in R$, the arc $[P,Q]$ is a straight line segment $($as
a subset of $\Lambda )$, and $Q$ is a turning point. Then $\ra p_{\h 0}
\prec \ra q_{\h 0}$.
\end{lem}

{\bf Arc-codes.} We call a word $\a$ an \emph{arc-code} if there
exists a basic arc $[E_j, E_{j+1}]$ such that $[E_j, E_{j+1}] =
I_{\a}$. Note that, by definition, every arc-code of length $\ge 1$
starts with $-$. Also, in this case, if $I_{\a}$ is a basic arc and
$|\a|$ is even (respectively odd), then $I_{\a} \subset R^+$
(respectively $I_{\a} \subset R^-$).

\begin{lem}\label{l:order of arc-codes}
Let $\a$, $\b$ be two different arc-codes and let $I_{\a}$, $I_{\b}$
be the corresponding basic arcs. If $\a$ and $\b$ have different
lengths, but $|\a|$ and $|\b|$ have the same parity, then $|\a| >|\b|$
if and only if the basic arc $I_{\a}$ is farther from $X$ then the
basic arc $I_{\b}$ $(d(I_{\a}, X) > d(I_{\b}, X))$. If $\a$ and $\b$
have the same length, then $\a \prec \b$ if and only if $d(I_{\a}, X)
> d(I_{\b}, X)$.
\end{lem}

\begin{proof}
If $\a$ and $\b$ have different lengths, but $|\a|$ and $|\b|$ have
the same parity, then we take $n$ of the same parity as $|\a|$ and
$|\b|$ and such that $L^{-n}(I_\a)$ and $L^{-n}(I_\b)$ are contained
in $[G_0,T_0]$. Choose $P\in L^{-n}(I_\a)$ and $Q\in L^{-n}(I_\b)$.
Compare $\ra p_{\h 0}$ with $\ra q_{\h 0}$. By the parity assumptions, 
they both start with the odd number of
$+$s. If $|\a| >|\b|$ then $\ra q_{\h 0}$ starts with more $+$s, so
$\ra q_{\h 0} \prec \ra p_{\h 0}$, and therefore $Q \tl P$. This means
that $I_{\b}$ is closer to $X$ than $I_{\a}$.

If $\a$ and $\b$ have the same length, we make the same construction.
Then $\a \prec \b$ is equivalent to $\ra q_{\h 0} \prec \ra p_{\h 0}$
(remember of the odd number of $+$s in front), and, as before,
$I_{\b}$ is closer to $X$ than $I_{\a}$.
\end{proof}

\section{Three approaches}\label{sec:three}

In this section we will present three approaches to coding the main
information about the unstable manifold $R$ of $X$, foldings of $R$
and the dynamics of $L$ on $R$. Those approaches will be via kneading
sequences, the folding pattern, and the folding tree.

{\bf Kneading sequences.} For each $n\in\Z$ the itinerary $\bar k^n$
of the $n$th turning point $T_n$ is a \emph{kneading sequence}. Let
$$\K := \{\bar k^n : n\in\Z\}$$
be the set of all kneading sequences of $L$. Similarly as for interval
maps, $\K$ contains the information about the basic properties
of $L$. Sometimes we will call $\K$ the \emph{kneading set}.

Strictly speaking, a turning point $T_n$ has two itineraries. They are
of the form $\lpinf \alpha^n \pm \cdot \ra k^n_{\h 0}$, where
$\alpha^n$ is the arc-code of the basic arc containing $L^{-2}(T_n)$.
Here for $\pm$ you can substitute any of $+$ and $-$. Therefore we can
think of this kneading sequence as a pair $(\alpha^n, \ra k^n_{\h
  0})$.

While $\K$ is only a set, we can recover the order in it by looking at
the arc-codes parts of the kneading sequences. Moreover, $\bar k^0$ is
the only kneading sequence with the arc-code part empty. Thus, given
an element $\bar k$ of $\K$ we can determine $n$ such that $\bar
k=\bar k^n$.

{\bf Folding pattern.} Write the sequence
$$(\dots, E_{-3}, E_{-2}, E_{-1}, E_0, E_1, E_2, E_3, \dots),$$
replacing each $E_i$ by the symbol $G$ if $E_i$ is a gluing point and
$T$ if $E_i$ is a turning or postturning point. Add additionally the
symbol $X$ between $E_0$ and $E_1$ (that is, where it belongs). We get
a sequence like
$$(\dots,T,G,T,G,T,G,X,T,G,T,G,T,T,T,G,T,\dots).$$
This is the \emph{folding pattern} of $L$.

The folding pattern carries some additional information, that we can
make visible. We know that $L$ restricted to $R$ is an orientation
reversing homeomorphism that fixes $X$. Moreover, it maps the set of
basic points bijectively onto the set of turning and postturning
points. Thus, we know which symbol of the folding pattern is mapped to
which one (see Figure~\ref{fig:fpmap}).

\begin{figure}[h]
        \unitlength=10mm
        \begin{picture}(20,3.6)(3,4.4)

	\put(8.5,6){\circle*{0.1}}
	\put(8.3,5.3){$X$}

	\put(3,6){\line(1,0){14}}

	\put(3,5.9){\line(0,1){0.2}}\put(2.8,5.3){$T$}
	\put(4,5.9){\line(0,1){0.2}}\put(3.8,5.3){$G$}
	\put(5,5.9){\line(0,1){0.2}}\put(4.8,5.3){$T$}
	\put(6,5.9){\line(0,1){0.2}}\put(5.8,5.3){$G$}
	\put(7,5.9){\line(0,1){0.2}}\put(6.8,5.3){$T$}
	\put(8,5.9){\line(0,1){0.2}}\put(7.8,5.3){$G$}
	\put(9,5.9){\line(0,1){0.2}}\put(8.8,5.3){$T$}
	\put(10,5.9){\line(0,1){0.2}}\put(9.8,5.3){$G$}
	\put(11,5.9){\line(0,1){0.2}}\put(10.8,5.3){$T$}
	\put(12,5.9){\line(0,1){0.2}}\put(11.8,5.3){$G$}
	\put(13,5.9){\line(0,1){0.2}}\put(12.8,5.3){$T$}
	\put(14,5.9){\line(0,1){0.2}}\put(13.8,5.3){$T$}
	\put(15,5.9){\line(0,1){0.2}}\put(14.8,5.3){$T$}
	\put(16,5.9){\line(0,1){0.2}}\put(15.8,5.3){$G$}
	\put(17,5.9){\line(0,1){0.2}}\put(16.8,5.3){$T$}

	\put(9,5){\line(0,-1){0.3}}
	\put(9,4.7){\line(-1,0){2}}
	\put(7,4.7){\vector(0,1){0.3}}

	\put(10,5){\line(0,-1){0.5}}
	\put(10,4.5){\line(-1,0){5}}
	\put(5,4.5){\vector(0,1){0.5}}

	\put(11,5){\line(0,-1){0.7}}
	\put(11,4.3){\line(-1,0){8}}
	\put(3,4.3){\vector(0,1){0.7}}

        \put(8,6.5){\line(0,1){0.3}}
	\put(8,6.8){\line(1,0){1}}
	\put(9,6.8){\vector(0,-1){0.3}}

	\put(7,6.5){\line(0,1){0.5}}
	\put(7,7){\line(1,0){4}}
	\put(11,7){\vector(0,-1){0.5}}

	\put(6,6.5){\line(0,1){0.7}}
	\put(6,7.2){\line(1,0){7}}
	\put(13,7.2){\vector(0,-1){0.7}}

	\put(5,6.5){\line(0,1){0.9}}
	\put(5,7.4){\line(1,0){9}}
	\put(14,7.4){\vector(0,-1){0.9}}

	\put(4,6.5){\line(0,1){1.1}}
	\put(4,7.6){\line(1,0){11}}
	\put(15,7.6){\vector(0,-1){1.1}}

	\put(3,6.5){\line(0,1){1.3}}
	\put(3,7.8){\line(1,0){14}}
	\put(17,7.8){\vector(0,-1){1.3}}

	\end{picture}

	\caption{The action of the map on the folding pattern.}
	\label{fig:fpmap}
\end{figure}
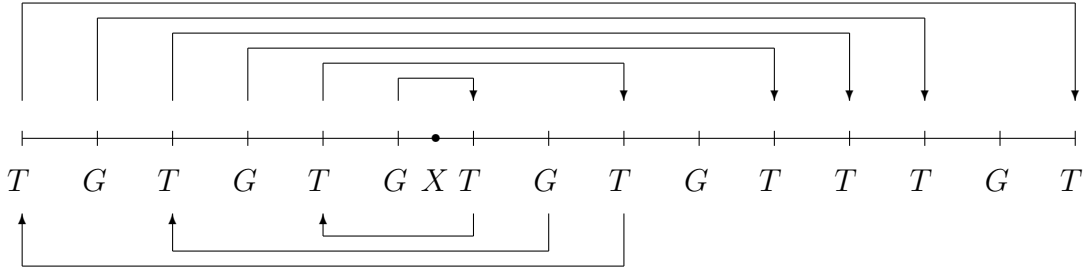

We know how to number the gluing points (the first to the left of $X$
is $G_0$). This, plus the information about the action of the map,
tells us which turning or postturning point is corresponding to a
given symbol $T$. Thus, we get $G$s and $T$s with subscripts and (some
of them) superscripts, like in Figure~\ref{fig.BP1}.

Another piece of information that we can read off the folding pattern
is which turning and postturning points and which basic arcs are in
the left or right half-plane. Namely, we know that the sign (which we
use for the itineraries) changes at every symbol $G$. Thus, we can
append our folding pattern with those signs and get a sequence like
this:
$$\dots -T-G+T+G-T-G+X+T+G-T-G+T+T+T+G-T- \dots\ .$$
For each symbol $T$ the signs adjacent to it from the left and right
are the same, so we can say that this is the sign of this $T$. Note
that it may happen that the corresponding postturning point is
actually on the $y$-axis, but still it has a definite sign.

Of course, we can put some of the additional information together, for
instance we can add to the folding pattern both the map and the signs
(see Figure~\ref{fig:fpmapsigns}).
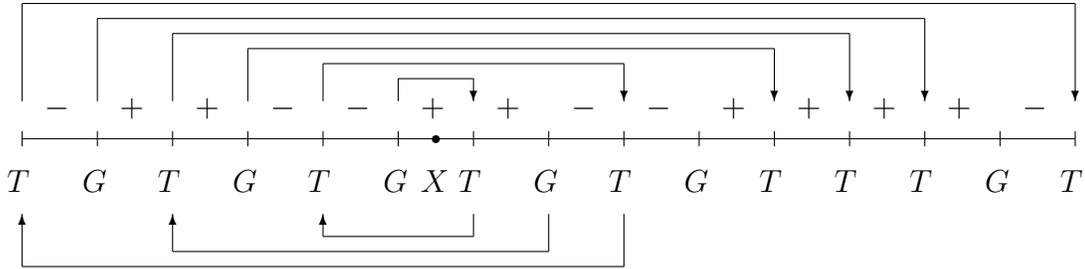
\begin{figure}[h]
        \unitlength=10mm
        \begin{picture}(20,3.6)(3,4.4)

	\put(8.5,6){\circle*{0.1}}
	\put(8.3,5.3){$X$}

	\put(3,6){\line(1,0){14}}

	\put(3,5.9){\line(0,1){0.2}}\put(2.8,5.3){$T$}
	\put(4,5.9){\line(0,1){0.2}}\put(3.8,5.3){$G$}
	\put(5,5.9){\line(0,1){0.2}}\put(4.8,5.3){$T$}
	\put(6,5.9){\line(0,1){0.2}}\put(5.8,5.3){$G$}
	\put(7,5.9){\line(0,1){0.2}}\put(6.8,5.3){$T$}
	\put(8,5.9){\line(0,1){0.2}}\put(7.8,5.3){$G$}
	\put(9,5.9){\line(0,1){0.2}}\put(8.8,5.3){$T$}
	\put(10,5.9){\line(0,1){0.2}}\put(9.8,5.3){$G$}
	\put(11,5.9){\line(0,1){0.2}}\put(10.8,5.3){$T$}
	\put(12,5.9){\line(0,1){0.2}}\put(11.8,5.3){$G$}
	\put(13,5.9){\line(0,1){0.2}}\put(12.8,5.3){$T$}
	\put(14,5.9){\line(0,1){0.2}}\put(13.8,5.3){$T$}
	\put(15,5.9){\line(0,1){0.2}}\put(14.8,5.3){$T$}
	\put(16,5.9){\line(0,1){0.2}}\put(15.8,5.3){$G$}
	\put(17,5.9){\line(0,1){0.2}}\put(16.8,5.3){$T$}

	\put(9,5){\line(0,-1){0.3}}
	\put(9,4.7){\line(-1,0){2}}
	\put(7,4.7){\vector(0,1){0.3}}

	\put(10,5){\line(0,-1){0.5}}
	\put(10,4.5){\line(-1,0){5}}
	\put(5,4.5){\vector(0,1){0.5}}

	\put(11,5){\line(0,-1){0.7}}
	\put(11,4.3){\line(-1,0){8}}
	\put(3,4.3){\vector(0,1){0.7}}

        \put(8,6.5){\line(0,1){0.3}}
	\put(8,6.8){\line(1,0){1}}
	\put(9,6.8){\vector(0,-1){0.3}}

	\put(7,6.5){\line(0,1){0.5}}
	\put(7,7){\line(1,0){4}}
	\put(11,7){\vector(0,-1){0.5}}

	\put(6,6.5){\line(0,1){0.7}}
	\put(6,7.2){\line(1,0){7}}
	\put(13,7.2){\vector(0,-1){0.7}}

	\put(5,6.5){\line(0,1){0.9}}
	\put(5,7.4){\line(1,0){9}}
	\put(14,7.4){\vector(0,-1){0.9}}

	\put(4,6.5){\line(0,1){1.1}}
	\put(4,7.6){\line(1,0){11}}
	\put(15,7.6){\vector(0,-1){1.1}}

	\put(3,6.5){\line(0,1){1.3}}
	\put(3,7.8){\line(1,0){14}}
	\put(17,7.8){\vector(0,-1){1.3}}

        \put(3.3,6.3){$-$}
        \put(4.3,6.3){$+$}
        \put(5.3,6.3){$+$}
        \put(6.3,6.3){$-$}
        \put(7.3,6.3){$-$}
        \put(8.3,6.3){$+$}
        \put(9.3,6.3){$+$}
        \put(10.3,6.3){$-$}
        \put(11.3,6.3){$-$}
        \put(12.3,6.3){$+$}
        \put(13.3,6.3){$+$}
        \put(14.3,6.3){$+$}
        \put(15.3,6.3){$+$}
        \put(16.3,6.3){$-$}

	\end{picture}

	\caption{The action of the map on the folding pattern with
          signs.}
	\label{fig:fpmapsigns}
\end{figure}

{\bf Folding tree.} We can think of the folding pattern as a countable
Markov partition for the map $L$ on $R$. Thus, we can consider the
corresponding Markov graph (the graph of transitions). The vertices of
this graph are the intervals $[E_i,E_{i+1}]$ (we write just the
corresponding number $i$ for them) and there is an arrow from $i$ to
$j$ if and only if $L([E_i,E_{i+1}])\supset [E_j,E_{j+1}]$. From the
folding pattern shown in Figure~\ref{fig:fpmap} we get the graph shown
in Figure~\ref{fig:ftarcs} (of course this tree goes down and is
infinite; we are showing only a part of it). This graph is almost a
tree, so we will call it the \emph{folding tree} of $L$. Except of $0$
and the arrows beginning at $0$, it is a subtree of the full binary
tree.

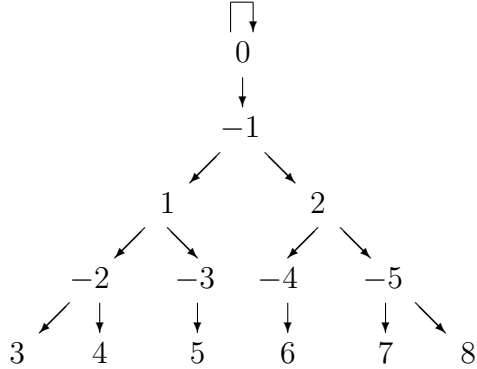
\begin{figure}[h]
        \unitlength=10mm
        \begin{picture}(20,6)(3.5,1.5)

          \put(10,6){$0$}
          \put(9.8,5){$-1$}
          \put(9,4){$1$}
          \put(11,4){$2$}
          \put(7.8,3){$-2$}
          \put(9.2,3){$-3$}
          \put(10.3,3){$-4$}
          \put(11.7,3){$-5$}
          \put(7,2){$3$}
          \put(8.1,2){$4$}
          \put(9.4,2){$5$}
          \put(10.6,2){$6$}
          \put(11.9,2){$7$}
          \put(13,2){$8$}

          \put(9.94,6.4){\line(0,1){0.4}}
          \put(9.94,6.8){\line(1,0){0.3}}
          \put(10.24,6.8){\vector(0,-1){0.4}}

          \put(10.1,5.8){\vector(0,-1){0.4}}
          \put(9.8,4.8){\vector(-1,-1){0.4}}
          \put(10.4,4.8){\vector(1,-1){0.4}}
          \put(8.8,3.8){\vector(-1,-1){0.4}}
          \put(11.4,3.8){\vector(1,-1){0.4}}
          \put(9.1,3.8){\vector(1,-1){0.4}}
          \put(11.1,3.8){\vector(-1,-1){0.4}}
          \put(7.8,2.8){\vector(-1,-1){0.4}}
          \put(12.4,2.8){\vector(1,-1){0.4}}
          \put(8.2,2.8){\vector(0,-1){0.4}}
          \put(9.5,2.8){\vector(0,-1){0.4}}
          \put(10.7,2.8){\vector(0,-1){0.4}}
          \put(12,2.8){\vector(0,-1){0.4}}

        \end{picture}

	\caption{A folding tree with numbers of basic arcs.}
	\label{fig:ftarcs}
\end{figure}

This tree is in a natural way divided into \emph{levels}. The number
$0$ is at level $0$, the number $-1$ is at level $1$, and in general,
if the path from $-1$ to $i$ has $n$ arrows then $i$ is at level
$n+1$. It is easy to see how the levels are arranged. Starting with
level $1$, negative numbers are at odd levels, ordered with their
moduli increasing from the left to the right. If level $n$ ends with
$-i$ then level $n+2$ starts with $-(i+1)$. Positive numbers are at
even levels, ordered in a similar way. Therefore, if we have the same
tree without the numbers, like in Figure~\ref{fig:ftnaked}, we know
where to put which number. Of course, we are talking about the tree
embedded in the plane, so the order of the vertices at each level is
given.

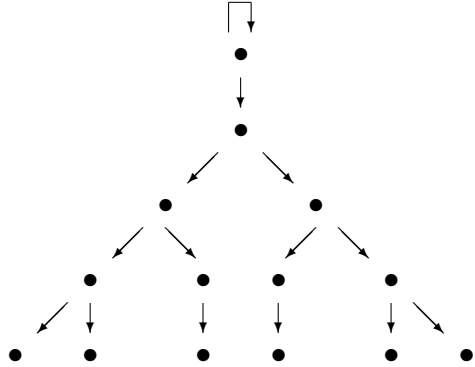
\begin{figure}[h!]
        \unitlength=10mm
        \begin{picture}(20,6)(3.5,1.5)

          \put(10,6){$\bullet$}
          \put(10,5){$\bullet$}
          \put(9,4){$\bullet$}
          \put(11,4){$\bullet$}
          \put(8,3){$\bullet$}
          \put(9.5,3){$\bullet$}
          \put(10.5,3){$\bullet$}
          \put(12,3){$\bullet$}
          \put(7,2){$\bullet$}
          \put(8,2){$\bullet$}
          \put(9.5,2){$\bullet$}
          \put(10.5,2){$\bullet$}
          \put(12,2){$\bullet$}
          \put(13,2){$\bullet$}

          \put(9.94,6.4){\line(0,1){0.4}}
          \put(9.94,6.8){\line(1,0){0.3}}
          \put(10.24,6.8){\vector(0,-1){0.4}}

          \put(10.1,5.8){\vector(0,-1){0.4}}
          \put(9.8,4.8){\vector(-1,-1){0.4}}
          \put(10.4,4.8){\vector(1,-1){0.4}}
          \put(8.8,3.8){\vector(-1,-1){0.4}}
          \put(11.4,3.8){\vector(1,-1){0.4}}
          \put(9.1,3.8){\vector(1,-1){0.4}}
          \put(11.1,3.8){\vector(-1,-1){0.4}}
          \put(7.8,2.8){\vector(-1,-1){0.4}}
          \put(12.4,2.8){\vector(1,-1){0.4}}
          \put(8.1,2.8){\vector(0,-1){0.4}}
          \put(9.6,2.8){\vector(0,-1){0.4}}
          \put(10.6,2.8){\vector(0,-1){0.4}}
          \put(12.1,2.8){\vector(0,-1){0.4}}

        \end{picture}

	\caption{A ``naked'' folding tree.}
	\label{fig:ftnaked}
\end{figure}

In a similar way as for the folding pattern, we can add some
information to the picture. The symbols $G$ and $T$ can be placed
between the vertices of the tree. The ones that are between the last
vertex of level $n$ and the first vertex of level $n+2$, will be
placed to the right of the last vertex of level $n$. The only
exception is $G_0$, which has to be placed to the left of the unique
vertex of level 1, in order to avoid a collision with other symbols.

We know which of the symbols are $G$s. By our construction, $G$s are
those elements of $\E$ that are in the interior of some
$L([E_i,E_{i+1}])$. This means that they are exactly the ones which
are between the siblings (vertices where the arrows from the common
vertex end). And once we have $G$s and $T$s marked, we can recover the
signs of the vertices, because we know that the signs change exactly at
$G$s. Then we get the folding tree marked as in
Figure~\ref{fig:ftall}.

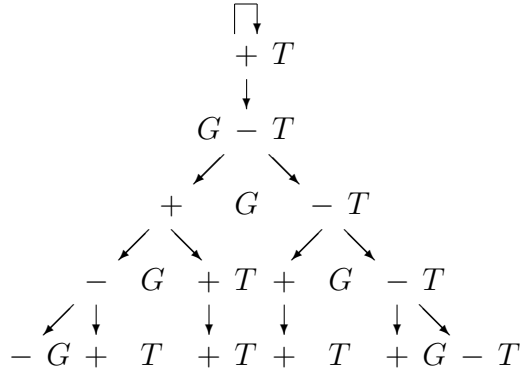
\begin{figure}[h!]
        \unitlength=10mm
        \begin{picture}(20,6)(3.5,1.5)

          \put(9.94,6){$+$}\put(10.44,6){$T$}
          \put(9.94,5){$-$}\put(10.44,5){$T$}\put(9.44,5){$G$}
          \put(8.94,4){$+$}\put(9.94,4){$G$}
          \put(10.94,4){$-$}\put(11.44,4){$T$}
          \put(7.94,3){$-$}\put(8.69,3){$G$}
          \put(9.44,3){$+$}\put(9.94,3){$T$}
          \put(10.44,3){$+$}\put(11.19,3){$G$}
          \put(11.94,3){$-$}\put(12.44,3){$T$}
          \put(6.94,2){$-$}\put(7.44,2){$G$}
          \put(7.94,2){$+$}\put(8.69,2){$T$}
          \put(9.44,2){$+$}\put(9.94,2){$T$}
          \put(10.44,2){$+$}\put(11.19,2){$T$}
          \put(11.94,2){$+$}\put(12.44,2){$G$}
          \put(12.94,2){$-$}\put(13.44,2){$T$}

          \put(9.94,6.4){\line(0,1){0.4}}
          \put(9.94,6.8){\line(1,0){0.3}}
          \put(10.24,6.8){\vector(0,-1){0.4}}

          \put(10.1,5.8){\vector(0,-1){0.4}}
          \put(9.8,4.8){\vector(-1,-1){0.4}}
          \put(10.4,4.8){\vector(1,-1){0.4}}
          \put(8.8,3.8){\vector(-1,-1){0.4}}
          \put(11.4,3.8){\vector(1,-1){0.4}}
          \put(9.1,3.8){\vector(1,-1){0.4}}
          \put(11.1,3.8){\vector(-1,-1){0.4}}
          \put(7.8,2.8){\vector(-1,-1){0.4}}
          \put(12.4,2.8){\vector(1,-1){0.4}}
          \put(8.1,2.8){\vector(0,-1){0.4}}
          \put(9.6,2.8){\vector(0,-1){0.4}}
          \put(10.6,2.8){\vector(0,-1){0.4}}
          \put(12.1,2.8){\vector(0,-1){0.4}}

        \end{picture}

	\caption{A folding tree with $G$s, $T$s and signs.}
	\label{fig:ftall}
\end{figure}

Now that we have our three objects, the kneading sequences, the
folding pattern, and the folding tree, we can prove that they carry
the same information.

\begin{theorem}\label{t:equivalent}
The set of kneading sequences, the folding pattern, and the folding
tree are equivalent. That is, given one of them, we can recover the
other two.
\end{theorem}

\begin{proof} {\it From the kneading set to the folding pattern.}
Suppose we know the set of the kneading sequences and we want to
recover the folding pattern. As we noticed when we defined the
kneading set, we know which kneading sequence is the itinerary of
which point $T_n$. We proceed by induction. First, we know that in
$[E_{-1},E_1]$ there are three points of $\E$ (and $X$), and that they
should be marked from the left to the right $T,G,X,T$. We also know
how they are mapped by $L$. Now suppose that we know the points of
$\E$ in $L^n([E_{-1},E_1])$, how they are marked, and how they are
mapped by $L$. Some of those points (on the left or on the right,
depending on the parity of $n$) are not mapped to the points of this
set. Then we map them to new points, remembering that $L(X)=X$ and
that $L$ is a homeomorphism of $R$ reversing orientation. Those new
points have to be marked as $T$, because $L$ maps $\E$ onto
$\T\cup\widehat{\T}$. Now we use our information about the kneading
sequences. They are the itineraries of the first images of the points
marked $G$, and since we know the action of $L$ on $\E\cap
L^n([E_{-1},E_1]$, this determines the signs of all points marked $T$
in the picture that we have at this moment. We know that the signs
change at each point marked $G$, so we insert such a point between
every pair of $T$s with opposite signs (clearly, there cannot be two
consecutive $G$s). In such a way we get the points of $\E$ in
$L^n([E_{-1},E_1])$, and the information how they are marked, and how
they are mapped by $L$. This completes the induction step.

{\it From the folding pattern to the folding tree.} This we described
when we were defining the folding tree.

{\it From the folding tree to the kneading set.} As we observed, given
a folding tree, we can add to it the information about the signs and
the positions of the $G$ and $T$ symbols. The turning points are the
symbols $T$ placed directly below $G$s, and additionally $T_0$ is the
only symbol in the zeroth row. Now for every $T$ which is a turning
point, we go down along the tree, reading the signs immediately to the
left of the symbols (see Figure~\ref{fig:ftall}). In such a way, we
get the corresponding right tail of the kneading sequence (the signs
do not change at $T$s, so the sign of the vertex immediately to the
left of a given $T$ is the same as the sign of the $x$-coordinate of
the corresponding turning or postturning point). Going back (up) is
even simpler, since in two steps we get to a vertex and just go up the
tree along the edges.
\end{proof}

\section{Symbolic dynamics}\label{sec:symbolic}

When we want to investigate the symbolic system obtained from a Lozi
map by taking the space of all itineraries and the shift on this
space, the basic thing is to produce a tool for checking whether a
given bi-infinite sequence is an itinerary of a point. Remember that we
called such a sequence \emph{admissible}. If a sequence is an
itinerary of a point of $R$, we call it \emph{$R$-admissible}.

Recall that when we considered the itineraries of points of $R$, we
removed some of them, as non-essential, and we were left with the
space $\Sigma_R^e$. We will call the elements of this space
\emph{essential $R$-admissible} sequences. From the definition it
follows that this space is $\sigma$-invariant.

In the case of all admissible sequences the situation is more
complicated. We do not know whether in order to get rid of the
unnecessary, non-essential, sequences it is enough to remove isolated
ones. Thus, we define the space $\Sigma_\Lambda^e$ as the closure of
$\Sigma_R^e$, and call the elements of $\Sigma_\Lambda^e$
\emph{essential admissible} sequences. As the closure of a
$\sigma$-invariant space, the space $\Sigma_\Lambda^e$ is also
$\sigma$-invariant.

We have to show that the essential admissible sequences suffice for
the symbolic description. We know this for $R$, that is, we know that
$\pi(\Sigma_R^e)=R$, but the analogous property for $\Lambda$ requires
some simple topological consideration. We also want to show that the
essential admissible sequences are really essential, that is, we
cannot remove any of them from our symbolic system. Note that by the
definition, $\Sigma_\Lambda^e$ is compact.

\begin{theorem}\label{essential}
We have $\pi(\Sigma_\Lambda^e)=\Lambda$, that is, every point of
$\Lambda$ has an itinerary that is an essential admissible sequence.
Moreover, it is the minimal set with this property. That is, each
compact subset $\Xi$ of $\Sigma_\Lambda$ such that $\pi(\Xi)=\Lambda$,
contains $\Sigma_\Lambda^e$.
\end{theorem}

\begin{proof}
Since the set $\Sigma_\Lambda^e$ is compact, so is
$\pi(\Sigma_\Lambda^e)$. It contains $\pi(\Sigma_R^e)=R$, which is
dense in $\Lambda$, so it is equal to $\Lambda$.

Now suppose that $\Xi\subset\Sigma_\Lambda$ is a compact set such that
$\pi(\Xi)=\Lambda$. The itineraries of all points of $R$ which have
only one itinerary have to belong to $\Xi$. By Lemma~\ref{approx},
the set of those itineraries is dense in $\Sigma_R^e$. Since $\Xi$ is
closed, we get $\Sigma_R^e\subset\Xi$, and then
$\Sigma_\Lambda^e\subset\Xi$.
\end{proof}

Now we go back to essential $R$-admissible sequences.

For a bi-infinite path in the folding tree (with signs of vertices),
we will call the corresponding bi-infinite sequence of signs the
\emph{sign-path}.

\begin{theorem}\label{t:admis}
Let $\bar p$ be a bi-infinite sequence of $+$s and $-$s. Then $\bar p$
is essential $R$-admissible if and only if it is a sign-path in the
folding tree.
\end{theorem}

\begin{proof}
Assume that $\bar p$ is essential $R$-admissible. This means that
there exists a point $P\in R$ with itinerary $\bar p$, and $\bar p$ is
not isolated. For every $n\in \Z$ there is a basic arc $B_n$ to which
$L^n(P)$ belongs. Then there is an arrow in the folding tree from the
vertex representing $B_n$ to the vertex representing $B_{n+1}$, so
$\bar p$ is a sign-path in the folding tree.

Now assume that $\bar p$ is a sign-path in the folding tree. If the
corresponding bi-infinite sequence of vertices is
$(B_n)_{n=-\infty}^\infty$, then the sequence
$$\left(\bigcap_{i=-n}^n L^{-i}(B_i)\right)_{n=0}^\infty
=\left(\bigcap_{i=0}^n L^{-i}(B_i)\right)_{n=0}^\infty$$ of
intervals of $R$ is a nested sequence of compact sets, so there exists
a point $P\in R$ such that $L^n(P)\in B_n$ for every $n\in\Z$. Thus,
$\bar p$ is the itinerary of $P$ and clearly it is not isolated, and
this proves that $\bar p$ is essential $R$-admissible.
\end{proof}

The above theorem shows that the folding tree of $L$ determines the
set of all essential $R$-admissible sequences. By
Theorem~\ref{t:equivalent}, the same can be said if we replace the
folding tree by the set of kneading sequences or by the folding
pattern. However, in order to mimic the kneading theory for unimodal
interval maps, we would like to have a more straightforward
characterization of all essential $R$-admissible sequences by the
kneading set.

First we have to remember that itineraries of all points of $R$ start
with $\lpinf$. Next thing that simplifies our task is that $R$ is
invariant for $L$, so the set of all essential $R$-admissible
sequences is invariant for $\sigma$. This means that apart of the
sequence $\lpinf\cdot\rpinf$ (which is $R$-admissible, because it is
the itinerary of $X$), we only need a tool of checking essential
$R$-admissibility of sequences of the form $\lpinf\cdot p_0p_1p_2\dots
=\lpinf\cdot\ra p_{\h 0}$, such that $p_0=-$.

\begin{theorem}\label{knead-R}
A sequence $\lpinf\cdot\ra p_{\h 0}$, such that $p_0=-$, is essential
$R$-admissible if and only if for every kneading sequence
$\lpinf\alpha\pm\cdot\ra k_{\h 0}$, such that $\alpha=p_0p_1\dots p_m$
for some $m$, we have $\sigma^{m+2}(\ra p_{\h 0})\preceq\ra k_{\h 0}$.
\end{theorem}

\begin{proof}
Assume first that a point $P\in R$ has the itinerary $\lpinf\cdot\ra
p_{\h 0}$, such that $p_0=-$, a turning point $Q$ has the itinerary
$\lpinf\alpha\pm\cdot\ra k_{\h 0}$, and $\alpha=p_0p_1\dots p_m$ for
some $m$. If $L^{m+2}(P)=Q$, then $\sigma^{m+2}(\ra p_{\h 0})=\ra
k_{\h 0}$. If $L^{m+2}(P)\neq Q$, then the arc $[L^{m+2}(P),Q]$ is a
straight line segment (as a subset of $\Lambda$), so by
Lemma~\ref{l:lessknead}, $\sigma^{m+2}(\ra p_{\h 0})\prec\ra k_{\h
  0}$.

Assume now that a sequence $\lpinf\cdot\ra p_{\h 0}$, such that
$p_0=-$, is given, and that for every kneading sequence
$\lpinf\alpha\pm\cdot\ra k_{\h 0}$, such that $\alpha=p_0p_1\dots p_m$
for some $m$, we have $\sigma^{m+2}(\ra p_{\h 0})\preceq\ra k_{\h 0}$.
Suppose that $\lpinf\cdot\ra p_{\h 0}$ is not essential
$R$-admissible. Then, by Theorem~\ref{t:admis}, it is not a sign-path
in the folding tree. This means that when we go down the tree trying
to find the corresponding sign-path, at a certain moment we get to a
vertex from which we cannot move down to a vertex with the correct
sign. That is, we found the part $\lpinf\cdot p_0p_1\dots p_n$ of a
sign-path, but we cannot append it with $p_{n+1}$.

Denote by $\widehat p_{n+1}$ the sign opposite to $p_{n+1}$. Look at
the basic arc $J:=I_{p_0p_1\dots p_n\widehat p_{n+1}}$ as a subset of
$\Lambda$. It is a straight line segment with endpoints that are
turning or postturning points. Consider the endpoint which is closer
to the $y$-axis. It is of the form $L^i(Q)$, where $Q$ is a turning
point and $i\in\N_0$. The kneading sequence of $Q$ is
$\lpinf\alpha\pm\cdot\ra k_{\h 0}$, with $\alpha=p_0p_1\dots p_m$,
where $m=n-i-1$. Thus, by the assumption, $\sigma^{n-i+1}(\ra p_{\h
  0})\preceq\ra k_{\h 0}$.

The point $Q$ is a turning point, so it is the right endpoint of
$L^{-i}(J)$. If the number of $+$s among $p_{n-i+1},p_{n-i+2},\dots,
p_n$ is even, then $L^i(Q)$ is the right endpoint of $J$, so $J$ is in
the left half-plane. This means that $\widehat p_{n+1}=-$, so
$p_{n+1}=+$. Both sequences $\sigma^{n-i+1}(\ra p_{\h 0})$ and $\ra
k_{\h 0}$ start with $p_{n-i+1},p_{n-i+2},\dots,p_n$. Then in
$\sigma^{n-i+1}(\ra p_{\h 0})$ we have $p_{n+1}=+$, while in $\ra
k_{\h 0}$ we have $\widehat p_{n+1}=-$. But this means that $\ra k_{\h
  0}\prec\sigma^{n-i+1}(\ra p_{\h 0})$, a contradiction.

Similarly, if the number of $+$s among $p_{n-i+1},p_{n-i+2},\dots,
p_n$ is odd, then $L^i(Q)$ is the left endpoint of $J$, so $J$ is in
the right half-plane. This means that $\widehat p_{n+1}=+$, so
$p_{n+1}=-$. Both sequences $\sigma^{n-i+1}(\ra p_{\h 0})$ and $\ra
k_{\h 0}$ start with $p_{n-i+1},p_{n-i+2},\dots,p_n$. Then in
$\sigma^{n-i+1}(\ra p_{\h 0})$ we have $p_{n+1}=-$, while in $\ra
k_{\h 0}$ we have $\widehat p_{n+1}=+$. But this means that $\ra k_{\h
  0}\prec\sigma^{n-i+1}(\ra p_{\h 0})$, a contradiction.

In both cases we got a contradiction, so $\lpinf\cdot\ra p_{\h 0}$ has
to be essential $R$-admissible.
\end{proof}

Now, that we know which sequences are essential $R$-admissible, we know
how to check whether a sequence $\bar p$ is essential admissible.
Namely, since the topology in the symbolic space is the product
topology, for each $n\in\N$ we check whether there is an essential
$R$-admissible sequence $\bar q$ such that $p_{-n}\dots
p_n=q_{-n}\dots q_n$. By Theorem~\ref{t:admis} this means that for
every $n$ we have to check whether the finite sequence $p_{-n}\dots
p_n$ is a finite sign-path in the folding tree.

\medskip
\noindent
Michal Misiurewicz\\
Department of Mathematical Sciences\\
Indiana University -- Purdue University Indianapolis\\
402 N.\ Blackford Street, Indianapolis, IN 46202\\
\texttt{mmisiurewicz@math.iupui.edu}\\
\texttt{http://www.math.iupui.edu/}$\sim$\texttt{mmisiure/}

\medskip
\noindent
Sonja \v Stimac\\
Department of Mathematics\\
University of Zagreb\\
Bijeni\v cka 30, 10\,000 Zagreb, Croatia\\
\& Department of Mathematical Sciences\\
Indiana University -- Purdue University Indianapolis\\
402 N.\ Blackford Street, Indianapolis, IN 46202\\
\texttt{sonja@math.hr}\\
\texttt{http://www.math.hr/}$\sim$\texttt{sonja}

\end{document}